\documentclass{amsart}

\usepackage{graphicx,amsmath,amsthm, amssymb,amsxtra,tikz,combelow,MnSymbol,pgfplots,hyperref, enumitem, appendix}

\pgfplotsset{compat = newest}
\pdfmapfile{+sansmathaccent.map}

\tikzset{anchorbase/.style={baseline={([yshift=-0.5ex]current bounding box.center)}}}
\tikzset{
    partial ellipse/.style args={#1:#2:#3}{
        insert path={+ (#1:#3) arc (#1:#2:#3)}
    }
}
\usetikzlibrary{calc}
\usetikzlibrary{decorations.markings}
\usetikzlibrary{decorations.pathreplacing}
\usetikzlibrary{arrows,shapes,positioning}
\tikzstyle directed=[postaction={decorate,decoration={markings,
    mark=at position #1 with {\arrow{>}}}}]
\tikzstyle rdirected=[postaction={decorate,decoration={markings,
    mark=at position #1 with {\arrow{<}}}}]
    
\usetikzlibrary{cd,external}

\definecolor{UCDblue}{cmyk}{1,0.56,0,0.34}
\definecolor{UCDgold}{cmyk}{0,0.19,1,0.15}

\DeclareMathOperator{\JM}{JM}

\newcommand{\C}{\mathbb{C}}

\newcommand{\MVcomment}[1]{} 

\usepackage[utf8]{inputenc}

\newtheorem{thm}{Theorem}[section]
\newtheorem{cor}[thm]{Corollary}

\newtheorem{lemma}[thm]{Lemma}
\newtheorem{conj}[thm]{Conjecture}

\theoremstyle{definition}

\newtheorem{defn}[thm]{Definition}

\theoremstyle{remark}
\newtheorem{rem}[thm]{Remark}

\title{Triply Graded Link Homology for Coxeter Braids on 4 Strands}
\author{Joshua P. Turner}
\address{Department of Mathematics
\\
University of California, Davis\\
Davis, CA 95616, USA}
\email{jpturner@ucdavis.edu}

\newcommand{\HHH}{\mathrm{HHH}}
\newcommand{\HY}{\mathrm{HY}}
\newcommand{\FT}{\mathrm{FT}}

\begin{document}

\begin{abstract}
    We compute the triply graded Khovanov-Rozansky homology for Coxeter braids on 4 strands.
\end{abstract}

\maketitle

\tableofcontents

\section*{Acknowledgments}

This work was partially supported by the NSF grant DMS-2302305. The author would like to thank Eugene Gorsky and Matt Hogancamp for useful discussions.

\section{Introduction}

In Section 6 of \cite{T}, the author uses a recursive process developed by Hogancamp and Elias in \cite{EH} to compute the triply-graded Khovanov Rozansky homology (also called HHH) for Coxeter braids on 3 strands, matching the result with the bigraded Hilbert series of a particular ideal in a polynomial ring. These braids were already known to be parity by work of Hogancamp in \cite{Hog1}.

In this paper, we use the same method to show that Coexter braids up to 4 strands are parity, and to compute $\HHH$. Previously, this method has been used by Hogancamp to compute $\HHH$ for $T(n,dn)$ torus links in \cite{Hog1}, and by Hogancamp and Mellit to compute $\HHH$ for all torus links in \cite{HogMel}.

\subsection{Homology of Coxeter links} 
The three gradings of $\HHH(\beta)$ are typically denoted $Q,T,$ and $A$. However, in line with \cite{EH}, we use the change of variables
$$q = Q^2, \quad t = T^2Q^{-2}, \quad a=AQ^{-2}.$$

\begin{defn}
    For any braid or braid-like diagram $\beta$, we say that $\HHH(\beta)$ is parity if it is supported in only even homological ($T$) degrees. We will also say $\beta$ itself is parity if $\HHH(\beta)$ is parity.
\end{defn}

Given integers $0 \leq d_1 \leq d_2 \leq \cdots \leq d_n$, we define the n-stranded braid

$$\beta(d_1, \dots, d_n) = \FT_n^{d_1} \FT_{n-1}^{d_2-d_1} \cdots \FT_2^{d_{n-1}-d_{n-2}}$$
$$= \JM_n^{d_1} \JM_{n-1}^{d_2} \cdots \JM_2^{d_{n-1}},$$
where $\FT_k$ and $\JM_k$ represent a full twist and a Jucys-Murphy element on the first $k$ strands respectively. These braids are part of the family of Coxeter braids defined in \cite{OR}. For the sake of comparison to the ideal $J$ defined in Section \ref{sec: intro ideal J}, we focus on pure Coxeter braids, whose closures have the maximum number of components. We expect that our calculations of $\HHH$ can be easily extended to all Coxeter braids as defined in \cite{OR}.

\begin{conj}
\label{conj: parity all n}
    The Coxeter braid $\beta(d_1, \dots, d_n)$ is parity for all $n$ and for all $0 \leq d_1 \leq \cdots \leq d_n$.
\end{conj}

In \cite{T}, the author shows that Coxeter braids on 3 strands are parity, and explicitly computes $\HHH^{a=0}(\beta)$. In Theorem \ref{thm: beta parity}, we show that Coxeter braids on 4 strands are also parity.

Note that the class of Coxeter links has some overlap with torus links. In particular we use the computation of $\HHH(\FT_4^k)=\HHH(T(k,4k))$ from \cite{Hog1} as a base case for our recursion (which corresponds to $d_1=d_2=d_3=d_4=k$), rather than re-derive it.


\subsection{Relation to the ideal $J$}
\label{sec: intro ideal J}

In this section, we use the $y$-ification of Khovanov-Rozansky homology, denoted by $\HY(\beta)$. It was defined in \cite{GH} where it was shown to be a topological invariant of the link obtained by closing $\beta$. We recall some results and conjectures describing $\HY(\beta)$ for Coxeter braids.

Define the ideal
    $$J(d_1, \dots, d_n) =  \bigcap_{i < j} (t_i-t_j, x_i - x_j)^{d_{i}} \subseteq \C[x_1, \dots, x_n, y_1, \dots, y_n].$$
    By the results of \cite{T}, such ideals are closely related to the equivariant Borel-Moore homology of certain affine Springer fibers. For $d_1=\ldots=d_n$, these ideals appear in the work of Haiman \cite{Haiman} on the geometry of Hilbert schemes of points on the plane.

Note that for both the ideal $J$ and the braid $\beta$, the largest $d_n$ is irrelevant, but we will continue including it for clarity.

If Conjecture \ref{conj: parity all n} holds, the following theorem of Gorsky and Hogancamp gives a description of the $a=0$ piece of both $\HY$ and $\HHH$ in terms of the ideal $J$.

    \begin{thm}[Gorsky, Hogancamp \cite{GH}]
\label{thm: HY(beta)}
    Assume that $\beta = \JM_n^{d_1} \ldots \JM_2^{d_{n-1}}$ and $\HHH^{a=0}(\beta)$ is parity. Then

    \begin{enumerate}
        \item $\HY^{a=0}(\beta) = \HHH^{a=0}(\beta) \otimes \C[y_1, \dots, y_n]$ and $\HHH^{a=0}(\beta) = \HY^{a=0}(\beta)/(y)$
        \item $\HY^{a=0}(\beta) = J(d_1,\ldots,d_n)$.
    \end{enumerate}
In particular, this also implies that the ideal $J(d_1,\ldots,d_n)$ is free over $\C[y_1,\ldots,y_n]$.  
\end{thm}

\begin{conj}
\label{conj: HY=J}
Suppose that $0 \leq d_1 \leq \cdots \leq d_1$, $\beta = \JM_n^{d_1} \ldots \JM_2^{d_{n-1}}$, and $J = J(d_1,\ldots,d_n)$. Then $\HHH^{a=0}(\beta)$ is parity and
$$
\HY^{a=0}(\beta)\simeq J
$$
and
$$\HHH^{a=0}(\beta)\simeq J/(y)J.$$
\end{conj}

This description doesn't guarantee that we can write out $\HHH^{a=0}(\beta)$ explicitly, as finding the bigraded Hilbert series for $J(d_1, \dots, d_n)$ is difficult in general. However, the recursive method we use to show that a braid is parity simultaneously gives a recursive way to compute $\HHH^{a=0}(\beta)$ as a rational function in $q$ and $t$.

In \cite{T}, the author shows that Coxeter links on 3 strands are parity and shows that Conjecture \ref{conj: HY=J} holds for $n=3$. In this paper, we use the same tools to show that the corresponding class of braids on 4 strands is also parity, and recursively calculate HHH in this case as well.

\begin{thm}
\label{thm: J free over y} 
For the braid $\beta = \JM_4^{d_1}\JM_3^{d_2} \JM_2^{d_3}$ with $0\leq d_1 \leq d_2 \leq d_3 \leq d_4$, and $J = J(d_1,d_2,d_3,d_4)$,
$$\HY^{a=0}(\beta) = J.$$

In particular, the ideal $J$ is free over as a module over $\C[y_1, \dots, y_4]$, and its bigraded Hilbert series is given by
$$\frac{1}{(1-t)^4}\HHH^{a=0}(\beta).$$
\end{thm}
\begin{proof}
    We show that the braid $\beta$ is parity in Theorem \ref{thm: beta parity}. So by Theorem \ref{thm: HY(beta)}, 
    $$\HY^{a=0}(\beta) = J(d_1,d_2,d_3,d_4) = \HHH^{a=0}(\beta) \otimes \C[y_1, \dots, y_n].$$
\end{proof}

For $n=3$, the bigraded Hilbert series for $J(d_1,d_2,d_3)$ is explicitly computed in general in \cite{T}, but for $n=4$ there is no similar calculation to compare to. However, the recursions in this paper can be used to compute this series explicitly for any particular $d_1,d_2,d_3,d_4$.

\begin{rem}
Note that  \cite[Conjecture 4.8]{T} predicts a closed formula for the Hilbert series as a sum of certain rational functions over 10 standard Young tableaux of size 4. We plan to verify this conjecture in a future work.
\end{rem}

We are optimistic that Theorem \ref{thm: HY(beta)} can be generalized to arbitrary $n$ using the same recursive process, proving Conjecture \ref{conj: HY=J}. This would also confirm that the ideal $J(d_1, \dots, d_n)$ is free over $y$'s, as conjectured by the author in \cite{T}, and give a recursive formula for its bigraded Hilbert series.

\section{Calculations}

Here we compute Khovanov-Rozansky homology of Coxeter braids on 4 strands. Let 
$$\beta = \beta(d_1,d_2,d_3,d_4) = (\FT_2)^{d_3-d_2} \cdot (\FT_3)^{d_2-d_1} \cdot (\FT_4)^{d_1},$$
where $\FT_i$ is the full twist on the first $i$ strands.
We will also use the Jucys-Murphy braids $\JM_2=\sigma_1^2$, $\JM_3=\sigma_2\sigma_1^2\sigma_2$, and $\JM_4 = \sigma_3\sigma_2\sigma_1^2\sigma_2\sigma_3$. Note that $\FT_n = \JM_2 \JM_3 \cdots \JM_n$.

We use the recursion for triply graded Khovanov-Rozansky homology from \cite{EH}, as described in \cite{GKS}. Here $K_n$ are certain complexes of Soergel bimodules. We refer to \cite{GKS} and references therein for all details. Figure \ref{fig: recursions} below shows the properties of $K_n$ that we will use in our recursions. 

Essentially, Figure \ref{fig: recursions}(a) says that we can insert a $K_1$ anywhere up to a grading shift. Figure \ref{fig: recursions}(b) says that $K_n$ absorbs crossings within the $n$ strands that it spans. Figure \ref{fig: recursions}(c) allows us to 'close up' a strand passing through a $K_n$ as long as there are no crossings in the way, and Figure \ref{fig: recursions}(d) is the main recursive step that allows $K_n$ to grow if there is a strand wrapping around it.

\begin{figure}[h!]
\begin{center}
\begin{tikzpicture}
\draw (-2,1.25) node {$(a)$};
\draw (-1,1.25) node {\Large $\frac{1}{1-q}$};
\draw (0,0)--(0,1)--(-0.4,1)--(-0.4,1.5)--(0.4,1.5)--(0.4,1)--(0,1);
\draw (0,1.5)--(0,2.5);
\draw (0,1.25) node {$K_1$};
\draw (1,1.25) node {=};
\draw (1.5,0)--(1.5,2.5);
\end{tikzpicture}
\quad
\begin{tikzpicture}
\draw (-1,1.25) node {(b)};
\draw  (-0.5,1)--(-0.5,1.5)--(0.5,1.5)--(0.5,1)--(-0.5,1);
\draw (-0.4,0)--(-0.4,1);
\draw (0,0.5) node {$\cdots$};
\draw (0.4,0)--(0.4,1);
\draw (0,1.25) node {$K_n$};
\draw (-0.4,1.5)--(-0.4,2.5);
\draw (0.4,1.5)--(0.4,2.5);
\draw (-0.2,2.5)..controls (-0.2,2) and (0.2,2)..(0.2,1.5);
\draw[white,line width=3] (-0.2,1.5)..controls (-0.2,2) and (0.2,2)..(0.2,2.5);
\draw (-0.2,1.5)..controls (-0.2,2) and (0.2,2)..(0.2,2.5);
\draw (1,1.25) node {=};
\draw  (1.5,1)--(1.5,1.5)--(2.5,1.5)--(2.5,1)--(1.5,1);
\draw (1.6,0)--(1.6,1);
\draw (2,0.5) node {$\cdots$};
\draw (2.4,0)--(2.4,1);
\draw (2,1.25) node {$K_n$};
\draw (1.6,1.5)--(1.6,2.5);
\draw (2.4,1.5)--(2.4,2.5);
\draw (2,2) node {$\cdots$};
\draw (3,1.25) node {=};
\draw  (3.5,1)--(3.5,1.5)--(4.5,1.5)--(4.5,1)--(3.5,1);
\draw (3.6,0)--(3.6,1);
\draw (4,2) node {$\cdots$};
\draw (4.4,0)--(4.4,1);
\draw (4,1.25) node {$K_n$};
\draw (3.6,1.5)--(3.6,2.5);
\draw (4.4,1.5)--(4.4,2.5);
\draw (3.8,1)..controls (3.8,0.5) and (4.2,0.5)..(4.2,0);
\draw[white,line width=3] (3.8,0)..controls (3.8,0.5) and (4.2,0.5)..(4.2,1);
\draw (3.8,0)..controls (3.8,0.5) and (4.2,0.5)..(4.2,1);
\end{tikzpicture}
\quad
\begin{tikzpicture}
\draw (-1,1.25) node {(c)};
\draw  (-0.5,1)--(-0.5,1.5)--(0.5,1.5)--(0.5,1)--(-0.5,1);
\draw (-0.4,0)--(-0.4,1);
\draw (-0.1,0.5) node {$\cdots$};
\draw (-0.1,2) node {$\cdots$};
\draw (0,1.25) node {$K_{n+1}$};
\draw (-0.4,1.5)--(-0.4,2.5);
\draw (0.4,1.5)..controls (0.4,2) and (1,2)..(1,1.5);
\draw (0.4,1)..controls (0.4,0.5) and (1,0.5)..(1,1);
\draw (0.2,0)--(0.2,1);
\draw (0.2,1.5)--(0.2,2.5);
\draw (1,1)--(1,1.5);
\draw (2,1.25) node {$=(t^n+a)$};
\draw  (3,1)--(3,1.5)--(4,1.5)--(4,1)--(3,1);
\draw (3.2,0)--(3.2,1);
\draw (3.8,0)--(3.8,1);
\draw (3.5,0.5) node {$\cdots$};
\draw (3.5,2) node {$\cdots$};
\draw (3.5,1.25) node {$K_{n}$};
\draw (3.2,1.5)--(3.2,2.5);
\draw (3.8,1.5)--(3.8,2.5);
\end{tikzpicture}
\end{center}

\begin{center}
\begin{tikzpicture}
\draw (-1.5,1.25) node {(d)};
\draw  (-0.5,1)--(-0.5,1.5)--(0.5,1.5)--(0.5,1)--(-0.5,1);
\draw (0,1.25) node {$K_{n}$};
\draw  (-0.7,2.5)--(-0.7,1);
\draw (-0.4,1.5)--(-0.4,2.5);
\draw (0.4,1.5)--(0.4,2.5);
\draw  (-0.7,1)..controls (-0.7,0.5) and (0.7,0.5)..(0.7,0);
\draw[white,line width=3] (-0.4,1)--(-0.4,0);
\draw [white,line width=3]  (0.4,1)--(0.4,0);
\draw  (-0.4,1)--(-0.4,-1);
\draw    (0.4,1)--(0.4,-1);
\draw [white,line width=3]    (0.7,0)..controls (0.7,-0.5) and (-0.7,-0.5)..(-0.7,-1);
\draw  (0.7,0)..controls (0.7,-0.5) and (-0.7,-0.5)..(-0.7,-1);

\draw (1,1.25) node {$=t^{-n}$};
\draw (1.7,2)..controls (1.5,1) and (1.5,1)..(1.7,0);
\draw  (2,1)--(2,1.5)--(3,1.5)--(3,1)--(2,1);
\draw (2.5,1.25) node {$K_{n+1}$};
\draw (2.1,1.5)--(2.1,2.5);
\draw (2.3,1.5)--(2.3,2.5);
\draw (2.9,1.5)--(2.9,2.5);
\draw (2.1,1)--(2.1,-1);
\draw (2.3,1)--(2.3,-1);
\draw (2.9,1)--(2.9,-1);
\draw [->] (3.5,1.25)--(4.3,1.25);
\draw (4.5,1.25) node {$q$};
\draw  (5,1)--(5,1.5)--(6,1.5)--(6,1)--(5,1);
\draw (5.5,1.25) node {$K_{n}$};
\draw (5.1,1.5)--(5.1,2.5);
\draw (5.9,1.5)--(5.9,2.5);
\draw (5.1,1)--(5.1,-1);
\draw (5.9,1)--(5.9,-1);
\draw (4.7,2.5)--(4.7,-1);
\draw (6.5,2)..controls (6.7,1) and (6.7,1)..(6.5,0);
\end{tikzpicture}
\end{center}
\caption{Recursions for $\HHH$}
\label{fig: recursions}
\end{figure}
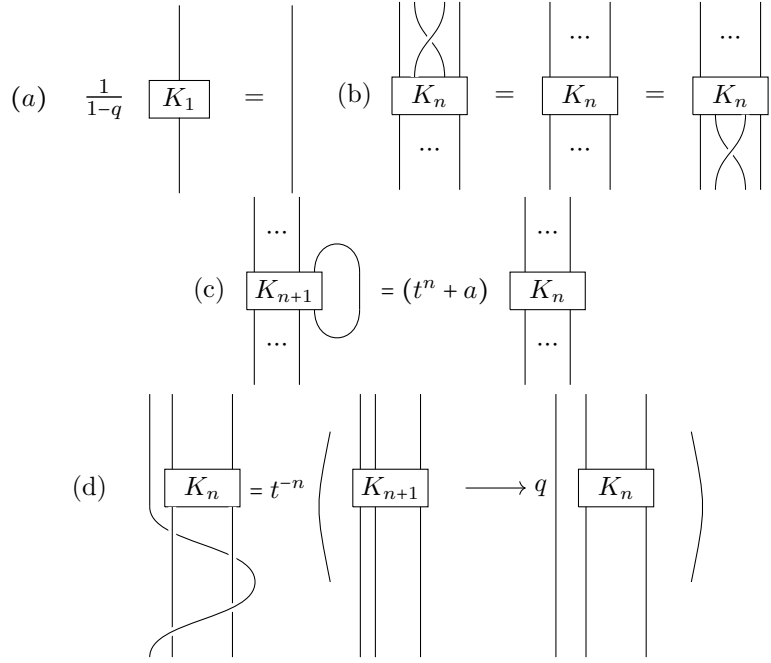

\begin{defn}
    For any braid or braid-like diagram $\beta$, we say that $\HHH(\beta)$ is parity if it is supported in only even homological ($T$) degrees, where $T^2 = \frac{t}{q}$.
\end{defn}

Note that $\HHH(\beta)$ must be parity for its graded Hilbert series to be rational in terms of $(q,a,t)$. 

\begin{rem}[\cite{GKS}]
    If both diagrams on the right-hand side of Figure \ref{fig: recursions}(d) are parity, then so is the left-hand side, and we can replace the map with addition on the level of rational functions (or a direct sum on the level of complexes). This implies that if we use Figure \ref{fig: recursions}(d) repeatedly to break down a braid $\beta$ into complexes that are known to be parity, then every complex along the way, including $\beta$ itself, is parity.
\end{rem}

We will make regular use of the fact that $$\FT_n = \JM_2\JM_3 \cdots \JM_n,$$ and that $\sigma_i$ commutes with 
$\FT_n$ for all $1 \leq i\leq n$ and with $\JM_n$ for $1<i<n$. In other words, the full twist commutes with all crossings on its strands, and the Jucys-Murphy braid commutes with internal crossings. Additionally, since $K_n$ absorbs crossings by Figure \ref{fig: recursions}(b), for $n \leq m$, we have that
$$K_n \FT_m = K_n \JM_{n+1} \cdots \JM_m.$$

We will also use the conjugation invariance of $\HHH$:
\begin{equation}
\label{eq: conj}
\HHH(\alpha\cdot \beta)=\HHH(\beta\cdot \alpha)
\end{equation}
for any braids $\alpha$ and $\beta$. We will denote the conjugate braids by $\alpha\beta\sim \beta\alpha$. We will additionally use the result from Hogancamp \cite{Hog1} that $\HHH(\FT_n^k)$ is parity for all $n,k \geq 0$.

Finally, note that Figure \ref{fig: recursions}(a), (b), and (d) are local, i.e. they hold on the level of homotopy equivalences of chain complexes. Braid conjugation and the closing up step (c) are non-local, and only hold on the level of $\HHH$.

\begin{defn}
Let
\begin{itemize}
    \item $A(n,m,l) = K_1\FT_2^{l}\FT_3^{m} \FT_4^{n}$,
    \item $B(n,m) = K_2 \FT_3^{m} \FT_4^{n}$,
    \item $C(n) = K_3 \FT_4^{n}$.
\end{itemize}
\end{defn}

In this notation,
$$K_1\beta(d_1,d_2,d_3,d_4) = A(d_1, d_2-d_1, d_3-d_2).$$

Assuming that all braids are parity, we can write Figure \ref{fig: recursions}(d) in the form
$$K_n \JM_{n+1} = t^{-n}K_{n+1} + qt^{-n}K_n.$$ 
By Figure \ref{fig: recursions}(b), $K_n \JM_{n+1} = K_n \FT_{n+1}$, so we immediately get the following recursions:

\begin{equation}
\label{A recursion}
A(n,m,l) = t^{-1} B(n,m) + qt^{-1}A(n,m,l-1)
\end{equation}

\begin{equation}
\label{B recursion}
B(n,m) = t^{-2} C(n) + qt^{-2}B(n,m-1)
\end{equation}

\begin{equation}
\label{C recursion}
C(n) = t^{-3}K_4 + qt^{-3}C(n-1)
\end{equation}

These recursions let us proceed by induction, as every application either decreases $n,m,$ or $l$ or shifts us from $A$ to $B$ or $B$ to $C$. Also note that they don't involve any closing up or conjugation, so they are all local. The majority of the work lies in the `base cases', when one or more of $l,m,n$ are 0. 

\subsection{Recursion for $C(n)$}
\label{sec: recursions for C}

Here we calculate $\HHH$ recursively for the braids $C(n)$ and for a twisted version $\sigma_3^{2k}C(n)$ that we will need in Section \ref{sec: recursions for B}.

\begin{lemma}
\label{lemma: C(n)}
    The braid $C(n)$ is parity for all $n \geq 0$, and
    $$\HHH(C(n))=t^{-3}\frac{1-(qt^{-3})^n}{1-qt^{-3}}\HHH(K_4)+\frac{(qt^{-3})^n(t^2+a)(t+a)(1+a)^2}{(1-q)}.$$
\end{lemma}
\begin{proof}
    The braid $C(0)$ has $K_3$ on the first 3 strands and an empty fourth strand. We can close up the $K_3$ using Figure \ref{fig: recursions}(c) and introduce a $K_1$ to close up the last strand, so
    $$\HHH(C(0)) = (t^2+a)(t+a)(1+a)^2/(1-q).$$
    We also have that
    $$\HHH(K_4) = (t^3+a)(t^2+a)(t+a)(1+a).$$
    So by induction, both terms on the right hand side of recursion \eqref{C recursion} are parity. Therefore $C(n)$ is parity, and can be computed recursively. 
\end{proof}

\begin{lemma}
    \label{lemma: twisted C(n)}
    The braid $\sigma_3^{2k}C(n)$ is parity for all $k,n \geq 0$, and
\begin{equation}
\label{eq: twisted C}
\HHH(\sigma_3^{2k}C(n))=t^{-3}\frac{1-(qt^{-3})^n}{1-qt^{-3}}\HHH(K_4)+(qt^{-3})^n(t^2+a)(t+a)(1-q)\HHH(\FT_2^k).
\end{equation}
\end{lemma}
\begin{proof}
    For $\sigma_3^{2k}C(0)$, we can close up (using Figure \ref{fig: recursions}(c)) the first two strands running through $K_3$, and then we are left with exactly $k$ full twists on the remaining two strands along with a $K_1$. So
    $$\HHH(\sigma_3^{2k}C(0)) = (t^2+a)(t+a)(1-q)\HHH(\FT_2^k).$$

    For $\sigma_3^{2k}C(n)$, we apply recursion \eqref{C recursion} locally to get
    $$\sigma_3^{2k}C(n) = t^{-3} \sigma_3^{2k}K_4 + qt^{-3}\sigma_3^{2k}C(n-1).$$
    The $\sigma_3$ crossings get absorbed  by $K_4$ (using Figure \ref{fig: recursions}(b)), so by induction on $n$, we conclude that $\sigma_3^{2k}C(n)$ is parity with the desired formula for $\HHH$.
\end{proof}

\subsection{Recursions for $B(n,m)$}
\label{sec: recursions for B}
Here we calculate $\HHH$ recursively for the braids $B(n,m)$ and for the twisted braids $\sigma_2^{2}B(n,m)$ that we will need in Section \ref{sec: recursions for A}.

\begin{lemma}
The braid $\sigma_3^{2k} B(0,0)$ is parity for all $k \geq 0$, and 
$$\HHH(\sigma_3^{2k} B(0,0)) = (t+a)(1+a)\HHH(\FT_2^{k}).$$
\end{lemma}
\begin{proof}
 Recall that $B(0,0)$ consists only of a $K_2$ on the left two strands, which does not overlap with $\sigma_3^{2k}$. 
 We close up the two strands running through $K_2$ (Figure \ref{fig: recursions}(c)) and are left with $k$ full twists on two strands.
\end{proof}

Now we consider the case $B(n,0)$. 

\begin{lemma}
    \label{lemma: B(n,0)}
    The braid $\sigma_3^{2k}B(n,0)$ is parity for all $n,k \geq 0$, and 
    $$\HHH\left[\sigma_3^{2k}B(n,0)\right] = t^{-2}\HHH\left[\sigma_3^{2k}C(n)\right] + qt^{-2}\HHH\left[\sigma_3^{2k+2}B(n-1,1)\right]$$
    $$= t^{-2}\HHH\left[\sigma_3^{2k}C(n)\right] + qt^{-4}\HHH \left[\sigma_3^{2k+2} C(n-1)\right] + q^2t^{-4}\HHH\left[\sigma_3^{2k+2} B(n-1,0)\right].$$

    In particular we have that $B(n,0)$ is parity for all $n \geq 0$.
\end{lemma}
\begin{proof}
    We have that locally,
    $$B(n,0) =  K_2 \JM_3 \JM_4 \FT_4^{n-1}$$
    $$= t^{-2}K_3\JM_4 \FT_4^{n-1} + qt^{-2}K_2\JM_4\FT_4^{n-1}$$
     \begin{equation}
    \label{eq: B(n,0) recursion 1}
    = t^{-2} C(n) + qt^{-2} K_2\JM_4\FT_4^{n-1}.
    \end{equation}
    Here we used Figure \ref{fig: recursions}(d) for $K_2\JM_3$, and the identity $K_3\JM_4=K_3\FT_4$ which follows from Figure \ref{fig: recursions}(b).
    Now note that $\JM_4 = \sigma_3 \JM_3 \sigma_3$, so
    $$\sigma_3^{2k}K_2 \JM_4 \FT_4^{n-1} = \sigma_3^{2k}K_2 \sigma_3 \JM_3 \sigma_3 \FT_4^{n-1} = \sigma_3^{2k}\sigma_3 K_2 \JM_3\FT_4^{n-1}\sigma_3,$$
    where we use the fact that $\sigma_3$ commutes with both $\FT_4$ and $K_2$.
    By conjugation invariance \eqref{eq: conj} we have that
    $$
    \HHH(\sigma_3^{2k}\sigma_3K_2 \JM_3 \FT_4^{n-1}\sigma_3)\simeq \HHH(\sigma_3^{2k+2}K_2 \JM_3 \FT_4^{n-1})= \HHH(\sigma_3^{2k+2} B(n-1,1)).
    $$
    Finally, a local application of recursion \eqref{B recursion} tells us that 
$$ \sigma_3^{2k+2} B(n-1,1) = t^{-2} \left[\sigma_3^{2k+2} C(n-1)\right] + qt^{-2}\left[\sigma_3^{2k+2} B(n-1,0)\right].$$ 

We know that $\sigma_3^{2k+2} C(n-1)$ is parity by Lemma \ref{lemma: twisted C(n)}, and we know that $\sigma_3^{2k+2n} B(0,0)$ is parity by Lemma \ref{lemma: twisted B(0,0)}. So by induction on $n$, we can conclude that $\sigma_3^{2k}B(n,0)$ is parity.
\end{proof}

\begin{cor}
\label{cor: B(n,m) parity}
    The braid $B(n,m)$ is parity for all $n,m \geq 0$.
\end{cor}
\begin{proof}
    We know that $B(n,0)$ is parity by Lemma \ref{lemma: B(n,0)} and that $C(n)$ is parity by Lemma \ref{lemma: C(n)}. So we use recursion \eqref{B recursion} and induction on $m$ to conclude that $B(n,m)$ is parity.
\end{proof}

Now we show that a twisted version $\sigma_2^2B(n,m)$ is also parity. We start with the step that requires the most caution to keep all braids parity.

\begin{lemma}
\label{lemma: twisted B(n,0)}
    If the right hand side is parity, then
    \begin{multline}
    \label{eq: twisted B(n,0)}
    \HHH\left[\sigma_2^2\sigma_3^{2k-2}B(n,0)\right] = t^{-2}\HHH\left[\sigma_3^{2k-2}C(n)\right] +\\ qt^{-4}\HHH \left[\widetilde{\JM}_3 \sigma_3^{2k-2} C(n-1) \right] + q^2t^{-4} \HHH\left[\sigma_2^2 \sigma_3^{2k} B(n-1,0)\right],
    \end{multline}
    where $\widetilde{\JM}_3=\sigma_3\sigma_2^2\sigma_3$ indicates the Jucys-Murphy element on the final 3 strands rather than the first 3 strands.
\end{lemma}
\begin{proof}
    Since the equation \eqref{eq: B(n,0) recursion 1} holds locally, we can multiply both sides on the left by $\sigma_2^2\sigma_3^{2k-2}$. For the first term in \eqref{eq: B(n,0) recursion 1}, if we include the additional crossings, note that 
    $$
    t^{-2}\left[\sigma_2^2 \sigma_3^{2k-2}C(n)\right]\sim t^{-2}\left[ \sigma_3^{2k-2}K_3\FT_4^n\sigma_2^2\right]=t^{-2}\left[ \sigma_3^{2k-2}K_3\sigma_2^2\FT_4^n\right]=
    t^{-2}\left[\sigma_3^{2k-2}C(n)\right].
    $$
    Here we conjugate $\sigma_2^2$ and pass it through $\FT_4^n$ until it is absorbed by $K_3$. Therefore
    $$
    t^{-2}\HHH\left[\sigma_2^2 \sigma_3^{2k-2}C(n)\right] = t^{-2}\HHH\left[\sigma_3^{2k-2}C(n)\right].$$

    For the second term in \eqref{eq: B(n,0) recursion 1}, we can again write $\JM_4 = \sigma_3 \JM_3 \sigma_3$, so
    $$\sigma_2^2 \sigma_3^{2k-2}K_2 \JM_4 \FT_4^{n-1} = \sigma_2^2 \sigma_3^{2k-2}K_2 \sigma_3 \JM_3 \sigma_3 \FT_4^{n-1} = \sigma_2^2 \sigma_3^{2k-1} K_2 \JM_3\sigma_3 \FT_4^{n-1}.$$
    Applying recursion \eqref{B recursion} gives
    \begin{equation}
    \label{eq: B(n,0) recursion 2}
    t^{-2}\left[\sigma_2^2 \sigma_3^{2k-1} K_3 \sigma_3 \FT_4^{n-1}\right] + qt^{-2} \left[ \sigma_2^2 \sigma_3^{2k-1} K_2 \sigma_3\FT_4^{n-1}\right].
    \end{equation}
    
    Now we slide $\sigma_3$ through $\FT_4$ and conjugate by it in the first term in \eqref{eq: B(n,0) recursion 2}, and slide it past $K_2$ in the second term in \eqref{eq: B(n,0) recursion 2}, which gives
    $$t^{-2}\left[\sigma_3 \sigma_2^2 \sigma_3^{2k-1} K_3 \FT_4^{n-1}\right] + qt^{-2} \left[ \sigma_2^2 \sigma_3^{2k} K_2\FT_4^{n-1}\right]$$
    $$= t^{-2} \left[\widetilde{\JM}_3 \sigma_3^{2k-2} C(n-1) \right] + qt^{-2} \left[\sigma_2^2 \sigma_3^{2k} B(n-1,0)\right].$$
\end{proof}
So in order to fully resolve $\sigma_2^2\sigma_3^{2k-2}B(n,0)$, we need still to resolve the cases $\widetilde{\JM}_3\sigma_3^{2k}C(n)$ and $\sigma_2^2 \sigma_3^{2k} B(0,0).$

\begin{lemma}
\label{lemma: JM C(n)}
    The braid $\widetilde{\JM}_3\sigma_3^{2k}C(n)$ is parity.
\end{lemma}
\begin{proof}
    
    Consider $n=0$ first. We close up the first strand (apply Figure \ref{fig: recursions}(c)), so
    $$\HHH(\widetilde{\JM}_3\sigma_3^{2k}C(0)) = 
    \HHH(\widetilde{\JM}_3\sigma_3^{2k}K_3)=
    (t^2+a)\HHH(\JM_3\sigma_2^{2k}K_2).$$
    Note that the first two braids here are on 4 strands, while the last is on 3 strands after the first strand was closed up.
    Now conjugate by $\JM_3$ and apply Figure \ref{fig: recursions}(d) to get 
    $$\JM_3\sigma_2^{2k}K_2 \sim \sigma_2^{2k}K_2\JM_3 = t^{-2}\left[\sigma_2^{2k}K_3\right] + qt^{-2}\left[\sigma_2^{2k}K_2 \right].$$
    The $K_3$ in the first term absorbs the extra crossings (Figure \ref{fig: recursions}(b)), and we can close up the first strand of the second term (Figure \ref{fig: recursions}(d)). So
    $$t^{-2}\left[\sigma_2^{2k}K_3\right] + qt^{-2}\left[\sigma_2^{2k}K_2 \right] = t^{-2}\left[K_3\right] + qt^{-2}(t+a)\left[K_1 \FT_2^k \right],$$
    where the final term is now on two strands. So we conclude that $\widetilde{\JM}_3\sigma_3^{2k}C(0)$ is parity, and
    
    $$\HHH(\widetilde{\JM}_3\sigma_3^{2k}C(0)) = t^{-2}(t^2+a)\left[(t^2+a)(t+a)(1+a) + q(t+a)(1-q)\HHH(\FT_2^k)\right].$$
    
    To resolve $\widetilde{\JM}_3\sigma_3^{2k}C(n)$, again we apply recursion \eqref{C recursion}. All extra crossings will be absorbed by $K_4$, so by induction on $n$ and the above base case, $\widetilde{\JM}_3\sigma_3^{2k}C(n)$ is parity.
\end{proof}

\begin{lemma}
    \label{lemma: twisted B(0,0)}
        The braid $\sigma_2^2 \sigma_3^{2k}B(0,0)$ is parity, and
        $$\HHH(\sigma_2^2\sigma_3^{2k} B(0,0)) = (t+a)t^{-1}\left[(t+a)(1-q)+q(1+a)\right]\HHH(\FT_2^{k}).$$
\end{lemma}
\begin{proof}
    By definition, $\sigma_2^2 \sigma_3^{2k}B(0,0)=\sigma_2^2 \sigma_3^{2k}K_2$.
    First, we close up one strand passing through $K_2$ on the left to obtain
    $$
    (t+a)\sigma_1^2 \sigma_2^{2k}K_1\sim (t+a) \sigma_2^{2k}K_1\sigma_1^2=(t+a) \sigma_2^{2k}K_1\JM_2.
    $$
    on three strands. Then we use Figure \ref{fig: recursions}(d) to get 
    $$
    \sigma_2^{2k}K_1\JM_2=t^{-1}[\sigma_2^{2k}K_2]+qt^{-1}[\sigma_2^{2k}K_1].
    $$
    For both terms, we can close up the first strand and have $k$ full twists remaining on the last two strands, getting
    $$t^{-1}\left[(t+a)(1-q)\FT_2^k + q(1+a)\FT_2^k\right].$$

\end{proof}

\begin{lemma}
\label{lemma: twisted B}
    The braid $\sigma_2^2B(n,m)$ is parity.
\end{lemma}
\begin{proof}
    We can see that $\sigma_2^2\sigma_3^{2k}B(n,0)$ is parity by Lemma \ref{lemma: twisted B(n,0)}, as all terms on the right hand side of \eqref{eq: twisted B(n,0)} are parity by Lemma \ref{lemma: JM C(n)} and Lemma \ref{lemma: twisted B(n,0)}.

    It follows that $\sigma_2^2B(n,m)$ is parity from induction on recursion \eqref{B recursion} and the above, as once again the extra crossings are absorbed by $K_3$.
\end{proof}







\subsection{Recursions for $A(n,m,l)$}
\label{sec: recursions for A}
    Now we know how to fully resolve all of the $B$'s and $C$'s, so the only thing left is $A$.

    \begin{lemma}
    \label{lemma: A(n,m,0)}
    If the right hand side is parity, then
        $$A(n,m,0) = t^{-1}B(n,m) + qt^{-1}\left[\sigma_2^2 A(n,m-1,1)\right]$$
        $$= t^{-1}B(n,m) + qt^{-2}\left[\sigma_2^2 B(n,m-1)\right] + q^2t^{-3}B(n,m-1) + q^3t^{-3}A(n,m-1,0).$$
    \end{lemma}
    \begin{proof}
    First, use Figure \ref{fig: recursions}(d) to get that
    $$A(n,m,0) = K_1 \FT_3^m \FT_4^n = K_1 \JM_2 \JM_3 \FT_3^{m-1} \FT_4^n$$
    \begin{equation}
    \label{eq: recursion A 1}
    = t^{-1}K_2 \JM_3\FT_3^{m-1} \FT_4^n + qt^{-1}K_1 \JM_3 \FT_3^{m-1} \FT_4^n.
    \end{equation}
    Since $K_2\JM_3 = K_2 \FT_3$, the first term in \eqref{eq: recursion A 1} is just $t^{-1} B(n,m)$. For the second term in \eqref{eq: recursion A 1}, note that $\JM_3 = \sigma_2 \JM_2 \sigma_2$ and that $\sigma_2$ commutes with $\FT_3$, $\FT_4$, and $K_1$. Sliding one $\sigma_2$ to the top and conjugating by it and moving the other $\sigma_2$ down past $K_1$, we get that
    $$K_1 \JM_3 \FT_3^{m-1} \FT_4^n = K_1 \sigma_2 \JM_2 \sigma_2 \FT_3^{m-1} \FT_4^n=\sigma_2K_1 \JM_2  \FT_3^{m-1} \FT_4^n\sigma_2$$
    $$\sim\sigma_2^2 K_1 \JM_2 \FT_3^{m-1}\FT_4^n = \sigma_2^2 A(m,n-1,1).$$
    Overall we have that
    $$\HHH[A(n,m,0)] = t^{-1}\HHH(B(n,m)) + qt^{-1}\HHH\left[\sigma_2^2 A(n,m-1,1)\right].$$
    If we apply recursion \eqref{A recursion} again, we get that
    $$\sigma_2^2 A(n,m-1,1) = t^{-1}\left[\sigma_2^2 B(n, m-1)\right] + qt^{-1} \sigma_2^2 A(n,m-1,0)$$
    $$= t^{-1}\left[\sigma_2^2 B(n, m-1)\right] + qt^{-1} A(n,m-1,1),$$
    noticing that 
        $$\sigma_2^2 A(n,m-1,0) =
        \sigma_2^2K_1\FT_3^{m-1}\FT_4^{n}=(\sigma_1\sigma_2)K_1\sigma_1^2\FT_3^{m-1}\FT_4^{n}(\sigma_1\sigma_2)^{-1}
        $$
        $$\sim K_1\sigma_1^2\FT_3^{m-1}\FT_4^{n} 
        = A(n,m-1, 1).$$
     Here we used that $\sigma_2 =(\sigma_1\sigma_2)\sigma_1(\sigma_1\sigma_2)^{-1}$ and the conjugating braid $(\sigma_1\sigma_2)$ commutes with $K_1,\FT_3$ and $\FT_4$.  
    Finally, we apply recursion \eqref{A recursion} one more time to $A(n,m-1,1)$ to get the desired result.
    \end{proof}
   We have already shown that most of the terms on the right hand side of the equation in Lemma \ref{lemma: A(n,m,0)} are parity. For the final base case, we refer to the work of Hogancamp in \cite{Hog1}, noting that the torus link T(4,4n) is the closure of the braid $\FT_4^n$.
    \begin{lemma} (Hogancamp \cite{Hog1})
    \label{lemma: A(n,0,0)}
        The braid $A(n,0,0) = (1-q)\FT_4^n$ is parity and $\HHH(A(n,0,0))$ can be computed recursively.
    \end{lemma}
    \begin{cor}
    \label{cor: A(n,m,0) parity}
        The braid $A(n,m,0)$ is parity for all $n,m \geq 0$.
    \end{cor}
    \begin{proof}
        This follows by induction from Lemma \ref{lemma: A(n,m,0)}, as we have shown that every $B$ term on the right hand side is parity (Lemma \ref{cor: B(n,m) parity} and Lemma \ref{lemma: twisted B}) and that the base case is parity (Lemma \ref{lemma: A(n,0,0)}).
    \end{proof}
    \begin{thm}
    \label{thm: beta parity}
        Assume that $0\leq d_1\leq d_2\leq d_3$. Then the braid 
        $$\beta = \beta(d_1,d_2,d_3,d_4) = (\FT_2)^{d_3-d_2} \cdot (\FT_3)^{d_2-d_1} \cdot (\FT_4)^{d_1}$$
        is parity, and $\HHH(\beta)$ can be computed using the recursive process laid out above.
    \end{thm}
    \begin{proof}
        We write $\HHH(\beta)=\frac{1}{1-q}\HHH(A(d_1, d_2-d_1, d_3-d_2)),$ so it is sufficient to prove that $A(n,m,l)$ is parity for all $n,m,l \geq 0$. Apply recursion \eqref{A recursion} repeatedly, reducing to terms of the form $B(n,m)$ and $A(n,m,0)$. These are parity by Corollary \ref{cor: B(n,m) parity} and Corollary \ref{cor: A(n,m,0) parity} respectively, and we can continue following the recursions laid out above to compute $\HHH(\beta)$.
    \end{proof}

\end{document}